\documentclass[a4paper,twoside]{amsart}
\usepackage[utf8]{inputenc}
\usepackage{amssymb,amsmath,amsthm,amscd}
\usepackage{mathrsfs}
\usepackage{enumerate}

\usepackage{lmodern}

\newcommand{\C}{\ensuremath{\mathbb{C}}}%
\newcommand{\R}{\ensuremath{\mathbb{R}}}%
\newcommand{\Z}{\ensuremath{\mathbb{Z}}}%
\newcommand{\sphere}{\ensuremath{\mathbb{S}}}%

\theoremstyle{definition}
\newtheorem{thm}{Theorem}[section]

\newtheorem{dfn}[thm]{Definition}
\newtheorem{lemma}[thm]{Lemma}
\newtheorem{prop}[thm]{Proposition}

\newtheorem{rem}[thm]{Remark}

\author{Tim de Laat}
\thanks{TdL is a Postdoctoral Fellow of the Research Foundation -- Flanders (FWO) and was partially supported by the Belgian Interuniversity Attraction Pole P07/18, and MdlS is partially supported by ANR grants OSQPI and NEUMANN}
\address{Tim de Laat, KU Leuven, Department of Mathematics
\newline Celestijnenlaan 200B -- box 2400, B-3001 Leuven, Belgium}
\email{tim.delaat@wis.kuleuven.be}

\author{Mikael de la Salle}
\address{Mikael de la Salle, CNRS-ENS de Lyon, UMPA UMR 5669
\newline F-69364 Lyon cedex 7, France
}
\email{mikael.de.la.salle@ens-lyon.fr}

\title[Noncommutative $L^p$-spaces and expanders]{Approximation properties for noncommutative $L^p$-spaces of high rank lattices and nonembeddability of expanders}

\begin{document}

\begin{abstract}
This article contains two rigidity type results for $\mathrm{SL}(n,\Z)$ for large $n$ that share the same proof. Firstly, we prove that for every $p \in [1,\infty]$ different from $2$, the noncommutative $L^p$-space associated with $\mathrm{SL}(n,\Z)$ does not have the completely bounded approximation property for sufficiently large $n$ depending on $p$.

The second result concerns the coarse embeddability of expander families constructed from $\mathrm{SL}(n,\Z)$. Let $X$ be a Banach space and suppose that there exist $\beta < \frac{1}{2}$ and $C>0$ such that the Banach-Mazur distance to a Hilbert space of all $k$-dimensional subspaces of $X$ is bounded above by $C k^\beta$. Then the expander family constructed from $\mathrm{SL}(n,\Z)$ does not coarsely embed into $X$ for sufficiently large $n$ depending on $X$.

More generally, we prove that both results hold for lattices in connected simple real Lie groups with sufficiently high real rank.
\end{abstract}

\maketitle

\section{Introduction}
The aim of this article is to prove two rigidity type results for $\mathrm{SL}(n,\mathbb{Z})$ for large $n$, and more generally for high rank lattices, that share the same proof.

Let $\Gamma$ be a countable discrete group, and let $L(\Gamma)$ be the group von Neumann algebra of $\Gamma$. For every $p \in (1,\infty)$, we can form the noncommutative $L^p$-space associated with $\Gamma$ by taking the completion of $L(\Gamma)$ with respect to the norm $\|.\|_p$ coming from the trace $\tau$ on $L(\Gamma)$ via $\|x\|_p=\tau((x^*x)^{\frac{p}{2}})^{\frac{1}{p}}$. A noncommutative $L^p$-space has the completely bounded approximation property (CBAP) if there exists a net of finite-rank maps on it that is uniformly bounded in the completely bounded norm and that approximates the identity map pointwise. We refer to Section \ref{sec=preliminaries} for more details on group von Neumann algebras, noncommutative $L^p$-spaces and the CBAP.

Until recently, no explicit examples of noncommutative $L^p$-spaces without the CBAP were known. The first explicit examples of such spaces were given by Lafforgue and the second named author in \cite{lafforguedelasalle}. They proved that for $n \geq 3$ and $p \in [1,\frac{4}{3}) \cup (4,\infty]$, the noncommutative $L^p$-spaces $L^p(L(\mathrm{SL}(n,\mathbb{Z})))$ do not have the CBAP. This result was extended in \cite{delaat1} and \cite{haagerupdelaat2}. From these two articles, it follows that for every lattice $\Gamma$ in a connected simple real Lie group with real rank at least $2$ and every $p \in [1,\frac{12}{11}) \cup (12,\infty]$, the space $L^p(L(\Gamma))$ does not have the CBAP. In fact, this is even true for $p \in [1,\frac{10}{9}) \cup (10,\infty]$, as follows from \cite[Appendix A]{delaatdelasalle}.

The first main result of this article deals with the CBAP for the noncommutative $L^p$-spaces associated with the group $\mathrm{SL}(n,\mathbb{Z})$.
\begin{thm} \label{thm=mainthmlp}
  Let $n \geq 3$, and let $r \geq 2n-3$. Then the noncommutative $L^p$-space $L^p(L(\mathrm{SL}(r,\mathbb{Z})))$ does not have the CBAP for $p \in \left[1,2-\frac 2 n\right) \cup \left(2+\frac 2 {n-2},\infty\right]$.
\end{thm}
Note that this result extends the result of Lafforgue and the second named author mentioned above. An analogue of Theorem \ref{thm=mainthmlp} in the non-Archimedean setting was already known from \cite{lafforguedelasalle}. Theorem \ref{thm=mainthmlp} was therefore expected. In fact, it answers one of the questions left open in \cite{lafforguedelasalle}. The following essential question remains open.

{\begin{center} {\bf{Question:}} Does $L^p(L(\mathrm{SL}(3,\Z)))$ have the CBAP for some $p \neq 2$?\end{center}}

The importance of this question is its relation with the (non-)isomorphism problem of the group von Neumann algebras of $\mathrm{PSL}(n,\Z)$ for different values of $n \geq 3$, which is a deep open problem going back to \cite{connes}. Indeed, an affirmative answer to the question above would imply that $L(\mathrm{SL}(3,\Z))$ is not isomorphic to $L(\mathrm{SL}(n,\Z))$ for certain values of $n \geq 4$.
\begin{rem}
  From Theorem \ref{thm=mainthmlp}, it follows that for every $p \neq 2$, the noncommutative $L^p$-space associated with any countable discrete group containing $\mathrm{SL}(n,\mathbb{Z})$ as a subgroup for every $n \geq 3$ does not have the CBAP. There are several ways to construct such a group. In particular, there are finitely presented examples \cite{chatterjikassabov}.
\end{rem}
The proof of Theorem \ref{thm=mainthmlp} proceeds in the same way as the proof of \cite[Theorem B]{lafforguedelasalle}, part of which was itself inspired by \cite{lafforguestrengthenedpropertyt}. However, the computations needed in this article are significantly more involved. The idea is as follows. Firstly, we use a result proved in \cite{lafforguedelasalle}, asserting that if $\Gamma$ is a lattice in a locally compact group $G$ and $L^p(L(\Gamma))$ has the CBAP for some $p \in (1,\infty)$, then $G$ has the so-called $\mathrm{AP}_{p,\mathrm{cb}}^{\textrm{Schur}}$ for that value of $p$ (see Section \ref{sec=preliminaries} for the definition of the $\mathrm{AP}_{p,\mathrm{cb}}^{\textrm{Schur}}$ and details). The $\mathrm{AP}_{p,\mathrm{cb}}^{\textrm{Schur}}$ was introduced in \cite{lafforguedelasalle} exactly to this purpose. The strategy then becomes to show the failure of the $\mathrm{AP}_{p,\mathrm{cb}}^{\textrm{Schur}}$ for $G$. The main new ingredient (Proposition \ref{prop=Schatten_class_estimates}) is a result on harmonic analysis on the sphere $\sphere^{n-1}$ for $n \geq 3$, for which a careful study of the spherical functions for the Gelfand pair $(\mathrm{SO}(n),\mathrm{SO}(n-1))$ is needed.

A more general version of Theorem \ref{thm=mainthmlp} for lattices in connected simple real Lie groups with high rank is obtained as well (see Theorem \ref{thm=mainthmlp_generalgroups}).\\

We now move to the second main result of this article. Let $S$ be a symmetric finite generating set of $\mathrm{SL}(n,\Z)$, and for $i \geq 1$, let $\pi_i \colon \mathrm{SL}(n,\Z) \rightarrow \mathrm{SL}(n,\Z / i\Z)$ denote the natural surjective homomorphism. As observed by Margulis, the Cayley graphs $(\mathrm{SL}(n,\Z/i\Z),\pi_i(S))_{i \geq 1}$ form an expander family (see Section \ref{sec=preliminaries} for the definition of expander family). It is an open problem whether for (say) $n=3$, this family embeds coarsely in any superreflexive Banach space (see \cite{lafforguestrengthenedpropertyt}, \cite{lafforguefastfourier}, \cite{pisiermemAMS}, \cite{mendelnaor}, \cite{delasalle1} for related results). A Banach space $X$ is superreflexive if every Banach space finitely representable in $X$ is reflexive. Our contibution to this question is that, modulo a classical open problem in Banach space theory, a superreflexive Banach space does not coarsely contain $(\mathrm{SL}(n,\Z/i\Z),\pi_i(S))_{i \geq 1}$ for $n$ large enough. In fact, we prove a non-embeddability result for these expander families in any (not necessarily superreflexive) Banach space satisfying a certain geometric criterion, to be made precise below. The notions from the geometry of Banach spaces that we use in what follows are recalled in Section \ref{subsec=geometrybanachspaces}.

For every Banach space $X$, consider the sequence of real numbers defined by
\[ d_k(X) = \sup\{ d(E,\ell^2_{\dim E}) \mid E \subset X, \dim E \leq k\},\]
where $d$ denotes the Banach-Mazur distance. It is always true that $d_k(X) \leq k^{\frac{1}{2}}$, and if $X$ has type $>1$ (in particular, if $X$ is superreflexive), then $d_k(X)= o(k^{\frac{1}{2}})$. Our results will apply to the Banach spaces $X$ for which
\begin{equation}\label{eq=d_kX_grows_slowly} \exists \beta < \frac 1 2,\,\exists C>0 \textrm{ such that }d_k(X) \leq C k^\beta\textrm{ for all }k\geq 1. \end{equation}
This includes the spaces of type $2$ and, more generally, the ones of type $p$ and cotype $q$ satisfying $\frac 1 p - \frac 1 q < \frac 1 2$. It is a well-known open problem whether all Banach spaces of type $>1$ satisfy \eqref{eq=d_kX_grows_slowly} (see Section \ref{sec=expanders}).

Our second result is as follows (see Theorem \ref{thm=expanders_general} for a more general statement for Schreier graphs coming from high rank lattices).
\begin{thm}\label{thm=expanders} Let $X$ be a Banach space satisfying \eqref{eq=d_kX_grows_slowly}. Then the expander family $(\mathrm{SL}(n,\Z/i\Z),\pi_i(S))_{i \geq 1}$ does not coarsely embed into $X$ for sufficiently large $n$ depending on $X$.
\end{thm}
First we prove that if $X$ satisfies \eqref{eq=d_kX_grows_slowly} and is superreflexive, then for $n$ sufficiently large, the group $\mathrm{SL}(n,\mathbb{R})$ has a version of property (T) relative to $X$ that was defined in \cite[Section 4]{lafforguestrengthenedpropertyt}. In order to prove this, we find a certain sequence of compactly supported measures $m_k$ on $\mathrm{SL}(n,\mathbb{R})$ such that $(\pi(m_k))_k$ converges for every isometric representation $\pi$ on $X$. The next step is to identify the limit of the sequence $(\pi(m_k))_k$ with a projection onto the $\pi(\mathrm{SL}(n,\mathbb{R}))$-invariant vectors. Here we cannot use the methods of \cite{lafforguestrengthenedpropertyt} (and hence we cannot prove Lafforgue's strong property (T) for $\mathrm{SL}(n,\mathbb{R})$ relative to $X$); instead we use a version of the Howe-Moore property for $\mathrm{SL}(n,\mathbb{R})$ as proved by Shalom (see \cite[Theorem 9.1]{BFGM}). This is where the superreflexivity assumption is used. Afterwards, we explain why the superreflexivity condition is, in fact, not necessary to get Theorem \ref{thm=expanders}, by showing that for a general Banach space $X$ satisfying \eqref{eq=d_kX_grows_slowly} and for $n$ sufficiently large, the group $\mathrm{SL}(n,\mathbb{R})$ has a version of property (T) with respect to a certain class of representations on $X$-valued $L^2$-spaces.

The version of property (T) discussed above passes from a locally compact group to its lattices \cite{lafforguestrengthenedpropertyt}, \cite{lafforguefastfourier}. Also, if a group has this property with respect to a superreflexive Banach space $X$, then the group has property (T$_X$) as defined in \cite{BFGM}. The following result is immediate.
\begin{thm}
  Let $X$ be a superreflexive Banach space satisfying \eqref{eq=d_kX_grows_slowly}. Then the groups $\mathrm{SL}(n,\mathbb{R})$ and $\mathrm{SL}(n,\mathbb{Z})$ have property (T$_X$) for sufficiently large $n$ depending on $X$.
\end{thm}
In fact, we prove that a ``non-uniform'' version of the above property is equivalent to property (T$_X$) (see Proposition \ref{prop=comparison_T}).\\

This article is organized as follows. After recalling some preliminaries in Section \ref{sec=preliminaries}, we obtain the aforementioned result on harmonic analysis on $\mathbb{S}^{n-1}$ in Section \ref{sec=sphere}. In Section \ref{sec=slr}, we use this result to prove Theorem \ref{thm=mainthmlp}. In Section \ref{sec=expanders}, we show how the proof of Theorem \ref{thm=mainthmlp} gives rise to Theorem \ref{thm=expanders}. We also include a result (Theorem \ref{thm=pisier}) of Gilles Pisier, relating the constant $d_k(X)$ to the relative Euclidean factorization constant $e_k(X)$ of a Banach space $X$ for all $k \geq 1$. This result is of independent interest as well.

\section*{Acknowledgements}
We thank Gilles Pisier for valuable discussions and for allowing us to include his proof of Theorem \ref{thm=pisier}. We also thank Stefaan Vaes and Alain Valette for useful remarks.

\section{Preliminaries} \label{sec=preliminaries}

\subsection{Group von Neumann algebras}
Let $\Gamma$ be a countable discrete group, and let $\lambda \colon \Gamma \to B(\ell^2(\Gamma))$ be its left regular representation, i.e., the representation of $\Gamma$ given by $(\lambda(g)\xi)(h)=\xi(g^{-1}h)$ for $g,h \in \Gamma$ and $\xi \in B(\ell^2(\Gamma))$. The group von Neumann algebra $L(\Gamma)$ of $\Gamma$ is given by the double commutant (in $B(\ell^2(\Gamma))$) of the set $\{ \lambda(g) \mid g \in \Gamma \}$. The von Neumann algebra $L(\Gamma)$ has a normal faithful trace $\tau$ given by $\tau(x)=\langle x\delta_1,\delta_1 \rangle$, where $\langle .,. \rangle$ denotes the inner product on $\ell^2(\Gamma)$. Group von Neumann algebras are important and motivating examples of von Neumann algebras.

\subsection{Noncommutative $L^p$-spaces and their approximation properties}
Let $M$ be a finite von Neumann algebra with normal faithful trace $\tau$. For $1 \leq p < \infty$, the noncommutative $L^p$-space $L^p(M,\tau)$ is the completion of $M$ with respect to $\|x\|_p=\tau((x^*x)^{\frac{p}{2}})^{\frac{1}{p}}$. For $p=\infty$, we set $L^{\infty}(M,\tau)=M$. In this article, we deal with noncommutative $L^p$-spaces coming from group von Neumann algebras.

Noncommutative $L^p$-spaces are important examples of operator spaces. The operator space structure on a noncommutative $L^p$-space $L^p(M,\tau)$ can be obtained by realizing $L^p(M,\tau)$ as an interpolation space of the couple $(M,L^1(M,\tau))$ (see \cite{kosaki}). An operator space $E$ has the completely bounded approximation property (CBAP) if there is a net $F_{\alpha}\colon E \rightarrow E$ of finite-rank maps with $\sup_{\alpha}\|F_{\alpha}\|_{cb} < \infty$ and $\lim_{\alpha} \|F_{\alpha}x-x\|=0$ for all $x \in E$. An operator space $E$ has operator space approximation property (OAP) if there exists a net $F_{\alpha}$ of finite-rank maps on $E$ such that $\lim_{\alpha} \|(\mathrm{id}_{\mathcal{K}(\ell^2)} \otimes F_{\alpha})x-x\|=0$ for all $x \in \mathcal{K}(\ell^2) \otimes_{\min} E$. If $E$ has the CBAP, it also has the OAP.

\subsection{The $\mathrm{AP}_{p,\mathrm{cb}}^{\textrm{Schur}}$} \label{subsec=apschur}
As mentioned in the introduction, we use the result of Lafforgue and the second named author that relates the CBAP of a noncommutative $L^p$-space to an approximation property of the underlying group. We recall this property below.

Recall that for a Hilbert space $\mathcal H$ and $p \in [1,\infty)$, the Schatten class $S^p(\mathcal H)$ is defined as the Banach space of bounded operators on $\mathcal H$ such that $\|T\|_p := Tr( |T|^p)^{1/p} <\infty$, and $S^{\infty}(\mathcal H)$ is the space $\mathcal{K}(\mathcal{H})$ consisting of compact operators. For a measure space $(X,\mu)$, the class $S^2(L^2(X,\mu))$ can be identified with $L^2(X \times X,\mu \otimes \mu)$. Hence, a function $\psi \in L^\infty(X \times X,\mu \otimes \mu)$ induces a bounded map on $S^2(L^2(X,\mu))$ corresponding to multiplication on $L^2(X \times X,\mu \otimes \mu)$. The function $\psi$ is said to be an $S^p$-multiplier if this map sends $S^p \cap S^2$ into $S^p$ and extends to a bounded map on $S^p$. The norm of this map will be denoted by $\|\psi\|_{M(S^p)}$ and its completely bounded norm by $\|\psi\|_{cbM(S^p)}$.

In the situation that $(X,\mu) = (G,m)$ is a locally compact group with left Haar measure, a function $\varphi \in L^\infty(G,m)$ is said to be an $S^p$-multiplier if the function $(g,h) \mapsto \varphi(g^{-1}h)$ is an $S^p$-multiplier. The corresponding bounded linear map on $S^p(L^2(G,m))$ is called $M_\varphi$. Its norm is denoted by $\|\varphi\|_{M(S^p)}$ and its completely bounded norm by $\|\varphi\|_{cbM(S^p)}$.

Recall that the Fourier algebra $A(G)$ (see \cite{eymard}) of a locally compact group $G$ consists of the coefficients of the left regular representation of $G$: we have $\varphi \in A(G)$ if and only if there exist $\xi,\eta \in L^2(G)$ such that for all $x \in G$ we have $\varphi(x)=\langle \lambda(x)\xi,\eta \rangle$. The norm given by $\|\varphi\|_{A(G)}=\min \{ \|\xi\|\|\eta\| \mid \forall x \in G \; \varphi(x)=\langle \lambda(x)\xi,\eta \rangle \}$, makes it into a Banach space.

Let $1 \leq p \leq \infty$. A locally compact group $G$ is said to have the $\mathrm{AP}_{p,\mathrm{cb}}^{\textrm{Schur}}$ if there exists a net $(\varphi_{\alpha})_{\alpha}$ in $A(G)$ such that $\sup_{\alpha} \|\varphi_{\alpha}\|_{cbMS^p(L^2(G))} < \infty$ and $\varphi_{\alpha} \to 1$ uniformly on compacta.

It is known (see \cite[Theorem 2.5]{lafforguedelasalle}) that if $\Gamma$ is a lattice in a locally compact group $G$, then for $1 \leq p \leq \infty$, the group $\Gamma$ has the $\mathrm{AP}_{p,\mathrm{cb}}^{\textrm{Schur}}$ if and only if $G$ has the $\mathrm{AP}_{p,\mathrm{cb}}^{\textrm{Schur}}$. It is also known that the $\mathrm{AP}_{p,\mathrm{cb}}^{\textrm{Schur}}$ passes to closed subgroups (see \cite[Proposition 2.3]{lafforguedelasalle}) and that it is preserved under local isomorphisms of Lie groups with finite center (see \cite[Proposition 3.11]{delaat1}).

We use the following result (see \cite[Corollary 3.13]{lafforguedelasalle}), relating the CBAP and the OAP to the $\mathrm{AP}_{p,\mathrm{cb}}^{\textrm{Schur}}$: if $p \in (1,\infty)$ and $\Gamma$ is a countable discrete group such that $L^p(L(\Gamma))$ has the OAP, then $\Gamma$ has the $\mathrm{AP}_{p,\mathrm{cb}}^{\textrm{Schur}}$ for that value of $p$. In particular, if $L^p(L(\Gamma))$ has the CBAP, then $\Gamma$ has the $\mathrm{AP}_{p,\mathrm{cb}}^{\textrm{Schur}}$.

We summarize the above results in the following lemma, which is exactly what we use in this article.
\begin{lemma} \label{lem=cbapoapapschur}
  Let $G$ be a locally compact group, and let $\Gamma$ be a lattice in $G$. If $p \in (1,\infty)$ and $G$ does not have the $\mathrm{AP}_{p,\mathrm{cb}}^{\textrm{Schur}}$ (for this $p$), then $L^p(L(\Gamma))$ does not have the CBAP or OAP.
\end{lemma}
For completeness, let us also mention a result from \cite{junge} (see also \cite[Theorem 4.2]{jungeruan}), asserting that the CBAP and OAP are equivalent for noncommutative $L^p$-spaces of a QWEP von Neumann algebra. This is not needed for what follows.

\subsection{Expander families}
The graphs we consider in this article are undirected, and we allow loops and multiple edges. Formally, this means that a graph is a pair $\mathcal{G}=(V,E)$ consisting of a set $V$ of vertices and a multiset $E$ of edges, i.e., an edge is a subset of $V$ of cardinality $1$ or $2$. The degree of a vertex $v \in V$ is the number of edges that contain $v$ (with this definition a loop counts as $1$ for the degree). A graph is said to be $k$-regular for some $k \geq 1$ if the degree of every vertex is $k$. The vertex set of a graph can be considered as a metric space with respect to the graph distance.

All graphs we study are finite Schreier graphs of finitely generated groups. They are constructed from a group $\Gamma$ with symmetric finite generating set $S$ and a finite index subgroup $\Lambda < \Gamma$, and denoted by $(\Gamma/\Lambda,S)$. The set of vertices is the coset space $\Gamma/\Lambda$ and an edge $\{v,w\}$ appears with multiplicity equal to the number of elements $s\in S$ such that $sv=w$ (this number is well-defined because $S$ is symmetric). This defines a $k$-regular graph, where $k$ is the number of elements of $S$. In the particular case when $\Lambda$ is a normal subgroup, the Schreier graph $(\Gamma/\Lambda,S)$ is actually a Cayley graph of the group $\Gamma/\Lambda$. There are no loops if $S \cap \Lambda = \emptyset$ and no multiple edges if $S$ maps injectively into the quotient $\Gamma/\Lambda$.

Let $\mathcal{G}=(V,E)$ be a finite graph. The boundary $\partial F$ of a set $F \subset V$ is defined by $\partial F = \{\{v,w\} \in E \mid v \in F, w \in V \setminus F\}$. The Cheeger constant $h(\mathcal{G})$ of the graph $\mathcal{G}$ is defined by
\[
    h(\mathcal{G}) = \min \left\{\frac{|\partial F|}{|F|} \bigg\vert F \subseteq V,\,0 < |F| \leq \frac{1}{2} |V| \right\}.
\]
A graph $\mathcal{G}$ is connected if and only if $h(\mathcal{G}) > 0$.

An expander family is a sequence of finite graphs with strong connectivity properties, which are quantified by the Cheeger constant.
\begin{dfn}
  Let $k \geq 1$, and let $\{\mathcal{G}_n\}_{n \geq 1}$ be a sequence of finite $k$-regular graphs. Then $\{\mathcal{G}_n\}_{n \geq 1}$ is an expander family if $|V(\mathcal{G}_n)| \to \infty$ as $n \to \infty$ and if there is a $\varepsilon > 0$ such that $h(\mathcal{G}_n) \geq \varepsilon$ for all $n \geq 1$. Here, $V(\mathcal{G}_n)$ denotes the vertex set of $\mathcal{G}_n$.
\end{dfn}
For more details on expander families, we refer to \cite{lubotzky}. Note that sometimes different conventions are used in the literature.

\subsection{Non-coarse-embeddability}\label{subsec=coarse_embedd}
Let us first recall the definition of coarse embedding. A family of graphs $X_i$ with induced distance $d_i$ embeds coarsely into a metric space $(Y,d)$ if there exist $1$-Lipschitz maps $f_i\colon X_i \to Y$ and an increasing map $\rho \colon [0,\infty) \to [0,\infty)$ such that $\lim_{t \to \infty} \rho(t) = \infty$ and $\rho(d_i(x,y)) \leq d(f_i(x),f_i(y))$ for all $i$ and all $x,y \in X_i$. We refer to \cite{nowakyu} for an overview of the theory.

In this article, we essentially work on the level of the Lie group $\mathrm{SL}(n,\R)$ rather than on the level of the Cayley graphs of $\mathrm{SL}(n,\Z/i\Z)$. The non-coarse-embeddability of the associated expander family follows, by an argument of Lafforgue \cite{lafforguestrengthenedpropertyt}, from a Banach space version of property (T) for $\mathrm{SL}(n,\R)$ (Theorem \ref{thm=TBanachiqueG}). Lafforgue's argument is an adaptation of the fact (due to Kazhdan and Margulis) that property (T) for $\mathrm{SL}(n,\R)$ implies that the Cayley graphs of $\mathrm{SL}(n,\Z/i\Z)$ form an expander family, and of Gromov's proof that such a family does not coarsely embed into a Hilbert space.

Let us recall Lafforgue's argument. Let $G$ be a locally compact group, let $\Gamma$ be lattice in $G$, let $(\Gamma_i)_i$ be a sequence of finite index subgroups in $\Gamma$ with index tending to $\infty$, and let $X$ be a superreflexive Banach space. Let $\pi_i$ denote the quasi-regular representation of $\Gamma$ on $\ell^2(\Gamma/\Gamma_i;X)$ and consider the direct sum $\pi_X = \oplus_i \pi_i$, which is an isometric representation of $\Gamma$ on $\ell^2(\sqcup_i \Gamma/\Gamma_i;X)$. Consider the induced representation $\mathrm{Ind}_\Gamma^G \pi_X$. Recall that the representation space of $\mathrm{Ind}_\Gamma^G \pi_X$ is the space $X'$ of Bochner-measurable functions $f \colon G\to \ell^2(\sqcup_i\Gamma/\Gamma_i;X)$ satisfying $\|f\| = \left(\int_{G/\Gamma} \|f(g)\|^2 dg\right)^{1/2}<\infty$ and $f(g\gamma) = \pi_X(\gamma^{-1}) f(g)$ for all $g \in G$ and $\gamma \in \Gamma$, and $G$ acts by 
\[ (\mathrm{Ind}_\Gamma^G \pi_X)(h) f(g) = f(h^{-1}g).\]
This construction depends on the choice of Haar measure on $G$, and we choose the one that is normalized so that $G/\Gamma$ has measure $1$.

The space $X'^G$ of invariant vectors for $\mathrm{Ind}_\Gamma^G \pi_X$ is the set of constant functions $f$ from $G$ to space of the invariant vectors for $\pi_X$, i.e., the space of $\xi \in \ell^2(\sqcup_i\Gamma/\Gamma_i;X)$ that are constant on each $\Gamma/\Gamma_i$. Also, $X'^G$ has an invariant complement subspace $Y'$, namely the space of functions $f \in X'$ with values in the space $Y:=\oplus_i \ell^2_0(\Gamma/\Gamma_i;X)$ consisting of $\xi \in \ell^2(\sqcup_i\Gamma/\Gamma_i;X)$ that have mean $0$ on each $\Gamma/\Gamma_i$. 

Let us assume that there exists a Borel measure $\nu$ on $G$ such that $\nu(1)=1$ and $\|\mathrm{Ind}_\Gamma^G \pi_X (\nu) \|_{B(Y')} \leq \frac{1}{4}$. The existence of such a $\nu$ is typically provided by some Banach space version of property (T) (see Section \ref{sec=expanders}). Let $\Omega \subset G$ be a Borel fundamental domain, i.e., for every $g \in G$, there is a unique $\gamma(g) \in \Gamma$ such that $g \in \Omega \gamma(g)$. The map from $Y$ to $Y'$ given by $f \mapsto \widetilde{f}$, where $\widetilde f(\omega \gamma) =  \pi_X(\gamma) f$ for $\omega \in \Omega$ and $\gamma \in \Gamma$, is an isometry. Also, the map from $Y'$ to $Y'$ given by $f \mapsto \int_\Omega f(\omega) d\omega$ has norm $1$. Hence, the composition of these maps with $\mathrm{Ind}_\Gamma^G \pi_X (\nu)$ has norm less than $\frac{1}{4}$ on $Y$ and is equal to $\pi_X(\mu)$ for the probability measure $\mu$ on $\Gamma$ satisfying $\mu(\gamma_0) = \int_\Omega \int_G 1_{\gamma(g^{-1} \omega) = \gamma_0} d\nu(g) d\omega$.
By replacing $\mu$ by a measure with finite support and at total variation distance less than $\frac{1}{4}$ from $\mu$, we may assume that $\|\pi_X(\mu)\| _{B(Y)}\leq \frac{1}{2}$. Hence, $\mathrm{Id} - \pi_X(\mu)$ is invertible on $Y$ and $\|(\mathrm{Id} - \pi_X(\mu))^{-1} \|\leq 2$. This means that for every $i$ and every $f_i \colon \Gamma/\Gamma_i \to X$ such that $\sum_{x \in \Gamma/\Gamma_i} f_i(x)=0$, we have
\begin{equation}
  \label{eq=poincare} \frac{1}{|\Gamma/\Gamma_i|}\sum_{x \in \Gamma/\Gamma_i} \|f_i(x)\|^2 \leq 4 \frac{1}{|\Gamma/\Gamma_i|} \sum_{x \in \Gamma/\Gamma_i} \left\lVert \int (f_i(x) -f_i(\gamma^{-1}x)) d\mu(\gamma) \right\rVert_X^2.
\end{equation}
It is classical that this inequality implies that if $S$ is a symmetric finite generating set in $\Gamma$, the family of graphs $(\Gamma/\Gamma_i,S)$ does not coarsely embed into $X$. Indeed, if the functions $f_i$ are $1$-Lipschitz, we have $\|f_i(x) - f_i(\gamma^{-1}x)\| \leq |\gamma|_S$ where $|\gamma|_S$ is the word-length of $\gamma$ with respect to the generating set $S$, so that the right-hand side is bounded by $4K^2$, where $K = \max\{ |\gamma|_S \mid \mu(\gamma)>0\}$. If moreover $f_i$ has mean zero (which can be achieved by subtracting from $f_i$ its average), by \eqref{eq=poincare} we get that $\|f_i(x)\|\leq 2\sqrt 2 K$ for at least half of the vertices in $\Gamma/\Gamma_i$. This prevents $(f_i)$ to be a coarse embedding. Indeed, since the graph has bounded degree, the typical distance between two points in $\Gamma/\Gamma_i$ is at least of order $\log(|\Gamma/\Gamma_i|)$, which tends to infinity.

For further use in the proof of Theorem \ref{thm=expanders}, let us observe that the above representation $\mathrm{Ind}_\Gamma^G \pi_X$ can be identified as the representation on $L^2(\sqcup_i(G \times_\Gamma \Gamma/\Gamma_i);X)$ coming from the measure-preserving action of $G$ on $\sqcup_i(G \times_\Gamma \Gamma/\Gamma_i)$. Here $G \times_\Gamma \Gamma/\Gamma_i$ is the quotient of $G \times \Gamma/\Gamma_i$ by the equivalence relation $(g,x) \sim (g\gamma,\gamma^{-1} x)$ for all $g \in G$, $\gamma \in \Gamma$ and $x \in \Gamma/\Gamma_i$.

\section{Harmonic analysis on the $(n-1)$-sphere} \label{sec=sphere}
Fix $n \geq 3$. In what follows, constants, functions and operators may implicitly depend on $n$. Equip the sphere $\sphere^{n-1} = \{x=(x_1,\dots,x_{n}) \in \R^{n} \mid \sum x_i^2 = 1\}$ with the Lebesgue probability measure. For $\delta \in [-1,1]$, let $T_\delta$ be the operator on $L^2(\sphere^{n-1})$ defined by $T_\delta f(x)=\textrm{the average of }f\textrm{ on }\{ y \in \sphere^{n-1} \mid \langle x,y\rangle = \delta\}$. Equivalently, considering $\mathrm{SO}(n-1)$ as the subgroup of $\mathrm{SO}(n)$ fixing the first coordinate vector $e_1$ and using the identification $\sphere^{n-1} \cong \mathrm{SO}(n-1) \backslash \mathrm{SO}(n)$ through the map $\mathrm{SO}(n-1) g \mapsto g^{-1} e_1$, we can consider $L^2(\sphere^{n-1})$ as a subspace of $L^2(\mathrm{SO}(n))$. Then $T_\delta$ is the operator on $L^2(\mathrm{SO}(n))$ equal to 
\begin{equation}\label{eq=def_Tdelta} \int_{\mathrm{SO}(n-1) \times \mathrm{SO}(n-1)} \lambda( u g u') du du' \in B(L^2(\mathrm{SO}(n)))\end{equation}
for $g \in \mathrm{SO}(n)$ satisfying $g_{11} = \delta$. Here, $\lambda$ denotes the left-regular representation.
\begin{prop}\label{prop=Schatten_class_estimates} For $|\delta|<1$, the operator $T_\delta$ belongs to $S^p(L^2(\sphere^{n-1}))$ if $p>2+\frac{2}{n-2}$. Moreover, for such $p$ there exist constants $C_p \geq 2$ and $\alpha_p \in (0,1)$ such that for all $\delta \in [-\frac{1}{2},\frac{1}{2}]$,
\[ \| T_0 - T_\delta\|_{S^p(L^2(\sphere^{n-1}))} \leq C_p |\delta|^{\alpha_p}.\]
\end{prop}
The case $n=3$ was proved in \cite[Lemma 5.3]{lafforguedelasalle}. For the proof of Proposition \ref{prop=Schatten_class_estimates}, we use some facts from the representation theory of $\mathrm{SO}(n)$ (see for example \cite[Section 7.2--7.4]{vandijk}) that we collect in the following lemma.
\begin{lemma} \label{lemma=vandijk} There is an orthogonal decomposition $L^2(\sphere^{n-1})= \oplus_{k=0}^\infty \mathcal H_k$ such that each $\mathcal H_k$ has finite dimension
\[ m_k = \frac{(n+k-3)!(n+2k-2)}{(n-2)! k!}.\]
Moreover, the operators $T_\delta$ are diagonal with respect to this decomposition, and 
$T_\delta\left|_{\mathcal H_k}\right. = \varphi_k(\delta) \mathrm{Id}_{\mathcal H_k}$, where $\varphi_k$ is given by the formula
\[ \varphi_k(x) = c_n \int_0^\pi (x+i\sqrt{1-x^2} \cos \varphi)^k \sin^{n-3} \varphi d\varphi\]
for $x \in [-1,1]$. Here, $c_n = \frac{\Gamma(\frac{n-1}{2})}{\sqrt \pi \Gamma(\frac{n-2}{2})}$, so that $\varphi_k(1)=1$.
\end{lemma}
\begin{rem} 
This lemma expresses the fact that $(\mathrm{SO}(n),\mathrm{SO}(n-1))$ is a Gelfand pair with spherical functions $g\mapsto \varphi_k(g_{11})$. For this Gelfand pair, these functions are Gegenbauer (also called ultraspherical) polynomials. The spaces $\mathcal H_k$ are distinct irreducible representations of $\mathrm{SO}(n)$, and the operators $T_\delta$ commute with the representation of $\mathrm{SO}(n)$, so that Schur's Lemma implies that they are diagonal in the decomposition $\oplus_k \mathcal H_k$. The value $\varphi_k(\delta)$ can be computed by considering the harmonic polynomial $(x_1+i x_2)^k \in \mathcal H_k$.
\end{rem}
\begin{lemma}\label{lemma=inequ_jacobi_pol} There exists a constant $C$ such that for all $k \geq 1$ and $x \in (-1,1)$,
\[ |\varphi_k(x)| \leq \frac{C}{(k(1-x^2))^{\frac{n-2}{2}}} \quad\textrm{ and }\quad |\varphi^\prime_k(x)| \leq \frac{C}{(k(1-x^2))^{\frac{n-2}{2}}} \frac{k}{\sqrt{1-x^2}}. \]
\end{lemma}
\begin{proof}
Since $\varphi_1(x) = x$, we can assume $k \geq 2$. We claim that there is a constant $C$ (depending on $n$) such that for $k \geq 1$,
\begin{equation}\label{eq=key_inequality_jacobi_pol} \int_0^\pi |x+i\sqrt{1-x^2} \cos \varphi|^k \sin^{n-3}\varphi d \varphi \leq \frac{C}{(k(1-x^2))^{\frac{n-2}{2}}}.\end{equation}
This implies the two inequalities of the lemma (for a different value of $C$). The first inequality is immediate (and holds already for $k \geq 1$), and for the second one, use
\begin{multline*} |\varphi^\prime_k(x)| = c_n \left|\int_0^\pi k\left(1-i \frac{x \cos \varphi}{\sqrt{1-x^2}}\right) (x+i\sqrt{1-x^2} \cos \varphi)^{k-1} \sin^{n-3} \varphi d\varphi\right| \\ \leq c_n \frac{k}{\sqrt{1-x^2}} \int_0^\pi |x+i\sqrt{1-x^2} \cos \varphi|^{k-1} \sin^{n-3}\varphi d \varphi.\end{multline*}Let us prove \eqref{eq=key_inequality_jacobi_pol}. Firstly, note that
\[ |x+i\sqrt{1-x^2} \cos \varphi|^2 = x^2 + (1-x^2) \cos^2\varphi = 1 - (1-x^2) \sin^2\varphi \leq e^{-(1-x^2) \sin^2\varphi}.\]
This implies
\[ \int_0^\pi |x+i\sqrt{1-x^2} \cos \varphi|^k \sin^{n-3}\varphi d \varphi \leq 2 \int_0^{\frac{\pi}{2}} e^{-\frac{k}{2}(1-x^2) \sin^2\varphi} \sin^{n-3}\varphi d\varphi.\]
Cut the integral into two pieces as $\int_0^{\frac{\pi}{2}} = \int_0^{\frac{\pi}{4}} + \int_{\frac{\pi}{4}}^{\frac{\pi}{2}}$. For $\varphi \in [\frac{\pi}{4},\frac{\pi}{2}]$, estimate $e^{-\frac{k}{2} (1-x^2) \sin^2\varphi} \sin^{n-3}\varphi$ by $e^{-\frac{k}{4}(1-x^2)}$ to dominate the second integral by $\frac{\pi}{4} e^{-\frac{k}{4}(1-x^2)}$. For the first integral, substitute $t= \sqrt{k(1-x^2)} \sin \varphi$ and use $d(\sin \varphi) = \cos \varphi d\varphi \geq \frac{1}{\sqrt 2} d\varphi$ to dominate the first integral by
\[ \frac{\sqrt{2}}{(k(1-x^2))^{\frac{n-2}{2}}}\int_0^{\sqrt{\frac{k}{2}(1-x^2)}} e^{-\frac{t^2}{2}} t^{n-3} dt.\]
These two inequalities together become
\begin{align*} \int_0^\pi |x+i\sqrt{1-x^2} &\cos \varphi|^k \sin^{n-3}\varphi d \varphi \\ &\leq \frac{2\sqrt{2}}{(k(1-x^2))^{\frac{n-2}{2}}} \int_0^{\infty} e^{-\frac{t^2}{2}} t^{n-3} dt + \frac{\pi}{2} e^{-\frac{k}{4}(1-x^2)},\end{align*}
which implies \eqref{eq=key_inequality_jacobi_pol}.
\end{proof}
\begin{proof}[Proof of Proposition \ref{prop=Schatten_class_estimates}]
By Lemma \ref{lemma=vandijk}, we have
\[ \| T_x\|_{S^p}^p = \sum_{k \geq 0} m_k |\varphi_k(x)|^p \quad\textrm{ and }\quad \|T_{0} - T_x\|_{S^p}^p = \sum_{k\geq 1} m_k |\varphi_k(0) - \varphi_k(x)|^p.\]
By the formula for $m_k$, there exists an $A > 0$ (depending on $n$) such that $m_k \leq A k^{n-2}$ for $k \geq 1$. Hence, by Lemma \ref{lemma=inequ_jacobi_pol}, we have $m_k |\varphi_k(x)|^p \leq C(x,n,p) k^{n-2 -p\frac{n-2}{2}}$ for some constant $C(x,n,p)$ depending on $x$, $n$ and $p$. We conclude that $T_x \in S^p$ if $p > 2+\frac{2}{n-2}$ and $x \in (-1,1)$, because the series $\sum_{k \geq 1} k^{n-2 -p\frac{n-2}{2}}$ converges if $n-2 -p\frac{n-2}{2}<-1$, i.e., $p> 2+\frac{2}{n-2}$. For the second estimate, assume that $x \in [-\frac{1}{2},\frac{1}{2}]$. Using Lemma \ref{lemma=inequ_jacobi_pol}, we dominate $|\varphi_k(0) - \varphi_k(x)|$ by $ |x|\sup_{y \in [0,x]} |\varphi_k'(y)|$ for small values of $k$, and by $|\varphi_k(0)|+|\varphi_k(x)|$ for large values of $k$. More precisely, we obtain a constant $C > 0$ (depending on $n$) such that for $k \geq 1$ and $x \in [-\frac{1}{2},\frac{1}{2}]$,
\[ |\varphi_k(x) - \varphi_k(0)|\leq  \frac{C}{k^{\frac{n-2}{2}}} \min\{1,k|x|\}. \]
The proposition now follows from a simple computation.
\end{proof}

\subsection{Consequences in terms of $S^p$-multipliers}\label{subsec_multipliers}
In Section \ref{sec=slr}, Proposition \ref{prop=Schatten_class_estimates} will be used in the following form.
\begin{lemma}\label{lemma=mutlipliers_SOn} Let $\varphi\colon\mathrm{SO}(n,\R) \to \C$ be a continuous $\mathrm{SO}(n-1)$-bi-invariant function that is a multiplier of $S^p(L^2(\mathrm{SO}(n)))$. If $g,g' \in \mathrm{SO}(n)$, then
\[ |\varphi(g) - \varphi(g')| \leq 2 C_p \max(|g_{11}|^{\alpha_p}, |g'_{11}|^{\alpha_p}) \|\varphi\|_{M(S^p)}.\]
\end{lemma}
\begin{proof}
Let $g,g^\prime \in \mathrm{SO}(n)$, and let $\delta = {g_{11}}$ and $\delta' = {g^\prime_{11}}$. If $\max(|\delta|,|\delta^\prime|) \geq \frac{1}{2}$, then $|\varphi(g) - \varphi(g')| \leq 2 \|\varphi\|_{L^\infty} \leq 2  \|\varphi\|_{M(S^p)}$, and the claim follows, since $2^{-\alpha_p}C_p \geq 1$. Therefore, the result follows from the following inequality: if $|\delta|,|\delta'| < 1$,
\[|\varphi(g) - \varphi(g')| \leq \|\varphi\|_{M(S^p)} \|T_{\delta} - T_{\delta'}\|_{S^p}.\]
We give two proofs of this inequality. Firstly, we consider again the operators $T_\delta$ on $\mathcal H= L^2(\mathrm{SO}(n))$ given by \eqref{eq=def_Tdelta}. Then for every $g \in \mathrm{SO}(n)$, we have $M_\varphi(T_{g_{11}}) = \varphi(g) T_{g_{11}}$. This equality is a particular case of a more general fact: for every $\mathrm{SO}(n-1)$-bi-invariant probability measure $\mu$ on $\mathrm{SO}(n)$ with support contained in $\{g \mid |g_{1,1}|<1\}$, we have $\lambda(\mu) \in S^p(L^2(\mathrm{SO}(n)))$ and $M_\varphi(\lambda(\mu)) = \lambda(\varphi \mu)$. If $\mu$ is absolutely continuous with respect to the Haar measure with a density in $L^2(\mathrm{SO}(n))$, then $\lambda(\mu) \in S^2(L^2(\mathrm{SO}(n)))$ and the equality $M_\varphi(\lambda(\mu)) = \lambda(\varphi \mu)$ is the very definition of $M_\varphi$ (this does not use that $\mu$ is $\mathrm{SO}(n-1)$-bi-invariant). The general case follows by a density argument, using that by Proposition \ref{prop=Schatten_class_estimates}, for any $\varepsilon>0$, the map $\mu \mapsto \lambda(\mu)$ is continuous from the set of $\mathrm{SO}(n-1)$-bi-invariant probability measures with support in $\{g \mid |g_{1,1}|\leq 1-\varepsilon\}$ equipped with the weak-* topology to $S^p(L^2(\mathrm{SO}(n)))$ equipped with the norm topology. The fact that the $T_\delta$'s have a common eigenvector with eigenvalue $1$, namely a constant function on $L^2(\mathrm{SO}(n))$, implies that
\[ |\varphi(g) - \varphi(g')| \leq \|\varphi(g)T_{\delta} - \varphi(g')T_{\delta'}\|_{S^p} \leq \|\varphi\|_{M(S^p)} \|T_{\delta} - T_{\delta'}\|_{S^p},\]
which proves the claim.

A dual proof proceeds along the lines of \cite[Section 3]{delaat1}. Let $q=\frac{p}{p-1}$ be the conjugate exponent of $p$. By \cite[Proposition 2.7]{delaat1}, we can write $\varphi(g) = \sum_k c_k m_k \varphi_k(g_{11})$, where the $c_k$'s play the role of Fourier coefficients and satisfy $(\sum_k m_k |c_k|^{q})^{\frac{1}{q}} \leq \|\varphi\|_{M(S^q)} = \|\varphi\|_{M(S^p)}$. Hence, by the H\"older inequality, 
\[ |\varphi(g) - \varphi(g') | \leq \|\varphi\|_{M(S^p)} \left(\sum_k m_k |\varphi_k(\delta) - \varphi_k(\delta^\prime)|^p\right)^{\frac{1}{p}} = \|\varphi\|_{M(S^p)} \|T_{\delta} - T_{\delta^\prime}\|_{S^p}.\]
\end{proof}

\section{Approximation properties for $L^p(L(\mathrm{SL}(r,\Z)))$.} \label{sec=slr}
In this section, we will prove Theorem \ref{thm=mainthmlp}. In fact, we will obtain a more general statement (Theorem \ref{thm=mainthmlp_generalgroups}), namely that certain noncommutative $L^p$-spaces associated with arbitrary lattices in connected high rank Lie groups do not have the CBAP. These results will follow from the following theorem.
\begin{thm} \label{thm=nonAP_high_rank} Let $n \geq 3$, let $r \geq 2n-3$ and $p \in \left[1,2-\frac 2 n\right) \cup \left(2+\frac 2 {n-2},\infty\right]$. Then there does not exist a sequence of functions $\varphi_n \in C_0(\mathrm{SL}(r,\R))$ such that $\sup_n \|\varphi_n\|_{M(S^p)} < \infty$ and
\[
  \lim_n \varphi_n(g) = 1\textrm{ for all }g \in \mathrm{SL}(r,\R).
\]
\end{thm}
Since for any locally compact group $G$, we have $A(G) \subset C_0(G)$, and since pointwise convergence is weaker than uniform convergence on compacta, the above result implies that $\mathrm{SL}(r,\R)$ does not have the $\mathrm{AP}_{p,\mathrm{cb}}^{\textrm{Schur}}$. As is mentioned in the introduction and more precisely in Lemma \ref{lem=cbapoapapschur}, a consequence of this is that for $p$ and $r$ as given in the theorem and a lattice $\Gamma$ in $\mathrm{SL}(r,\mathbb{Z})$, the noncommutative $L^p$-space $L^p(L(\Gamma))$ does not have the CBAP, i.e., Theorem \ref{thm=mainthmlp} follows directly.

The strategy of proving Theorem \ref{thm=nonAP_high_rank} is based on the approach of \cite{lafforguedelasalle} and \cite[Section 5]{haagerupdelaat}. Firstly, Proposition \ref{prop=Schatten_class_estimates} gives rise to certain local H\"older continuity estimates for $\mathrm{SO}(n)$-bi-invariant $S^p$-multipliers on $\mathrm{SL}(n,\R)$, as given in Lemma \ref{lemma=multipliers_SLn}. The next step is to find a path going to infinity in the Weyl chamber of $\mathrm{SL}(r,\R)$ for $r$ large enough by combining such local estimates. It turns out that $r=2n-3$ is enough. We now make this precise.

Fix $n \geq 3$. By embedding $\mathrm{SL}(2n-3,\R)$ into $\mathrm{SL}(r,\R)$ for $r\geq 2n-3$, we see that it is enough to consider the case $r = 2n-3$. Also, by duality, we can assume that $p>2+\frac{2}{n-2}$. Then the Theorem \ref{thm=nonAP_high_rank} follows from an averaging argument and Proposition \ref{prop=invariant_multiplier_estimate} below.

For $t,u,v \in \R$ with $t+\frac{u}{n-2}+v=0$, we use the notation 
\begin{equation}\label{eq=def_Duvt}D(v,u,t) = \mathrm{diag}(e^v,\dots,e^v,e^u,e^{t}\dots, e^{t}) \in \mathrm{SL}(2n-3,\R)\end{equation} for the diagonal matrix with $n-2$ diagonal entries equal to $e^v$, $1$ diagonal entry equal to $e^u$ and $n-2$ diagonal entries equal to $e^t$.
\begin{prop}\label{prop=invariant_multiplier_estimate} For $p>2+\frac{2}{n-2}$, there is a function $\varepsilon_p \in C_0(\R_+)$ such that for every $\mathrm{SO}(2n-3)$-bi-invariant multiplier $\varphi\colon \mathrm{SL}(2n-3,\R) \rightarrow \C$ of ${S^p(L^2(\mathrm{SL}(2n-3,\R)))}$, the function $\varphi(D(t,0,-t))$ has a limit $c$ for $t \to \infty$, and 
\[ |\varphi(D(t,0,-t)) - c | \leq \varepsilon_p(t) \|\varphi\|_{M(S^p)}.\]
\end{prop}
The crucial step to prove this proposition is the following lemma.
\begin{lemma}\label{lemma=multipliers_SLn} Let $\varphi\colon \mathrm{GL}(n,\R) \rightarrow \C$ be a multiplier of $S^p(L^2(\mathrm{GL}(n,\R)))$ that is $\mathrm{SO}(n)$-bi-invariant, and let $t<u<v \in \R$. Then for $0<\delta<u-t$, we have
\begin{align*}|\varphi(\mathrm{diag}&(e^v,e^u,e^t,\dots,e^t)) - \varphi(\mathrm{diag}(e^{v+\delta},e^{u-\delta},e^t,\dots,e^t)) | \\ &\leq 2 C_p e^{-\alpha_p(u-t-\delta)} \|\varphi\|_{M(S^p)}.\end{align*}
\end{lemma}
\begin{proof}
Let $u'=u-\delta$ and $v'=v+\delta$, and let $s = v+u-t$. Then $u,v,u',v' \in (t,s)$ and $u+v=u'+v'=s+t$. Consider the matrix $D= \mathrm{diag}(e^{\frac{s}{2}},e^{\frac{t}{2}},\dots,e^{\frac{t}{2}}) \in \mathrm{GL}(n,\R)$. The map $g \in \mathrm{SO}(n) \mapsto \varphi(D g D)$ is an $\mathrm{SO}(n-1)$-bi-invariant multiplier of $S^p(L^2(\mathrm{SO}(n)))$ of norm less than $\|\varphi\|_{M(S^p)}$, so that by Lemma \ref{lemma=mutlipliers_SOn},
\[ |\varphi(DgD) - \varphi(Dg'D)| \leq 2 C_p \max(|g_{11}|^{\alpha_p}, |g'_{11}|^{\alpha_p}) \|\varphi\|_{M(S^p)}.\]
Let now $g$ (resp.~$g'$) be a rotation of angle $\theta$ (resp.~$\theta'$) in the plane generated by the two first coordinate vectors of $\R^n$, so that $g_{11} = \cos \theta$ and $g'_{11} = \cos \theta'$. Then $\varphi(DgD) = \varphi(\mathrm{diag}(e^x,e^y,e^t,\dots,e^t))$ where $x\geq y$ are determined by
\[ \begin{pmatrix} e^{\frac{s}{2}} & 0 \\ 0 & e^{\frac{t}{2}}\end{pmatrix} \begin{pmatrix} \cos \theta & -\sin \theta \\ \sin \theta & \cos \theta\end{pmatrix} \begin{pmatrix} e^{\frac{s}{2}} & 0 \\ 0 & e^{\frac{t}{2}}\end{pmatrix} \in \mathrm{SO}(2) \begin{pmatrix} e^{x} & 0 \\ 0 & e^{y}\end{pmatrix} \mathrm{SO}(2).\]
By a simple computation (see also \cite[Lemme 2.8]{lafforguestrengthenedpropertyt} or \cite[Lemma 5.5]{haagerupdelaat}), there is a $\theta$ such that $x=v$, $y=u$ and $|\cos \theta| \leq e^{v-s} = e^{t-u}$. Similarly, there is $\theta'$ such that $|\cos \theta'|\leq e^{v+\delta - s} = e^{t+\delta - u}$ and such that $\varphi(D g' D) = \varphi(\mathrm{diag}(e^{v+\delta},e^{u-\delta},e^t,\dots,e^t))$. This proves the lemma.
\end{proof}
\begin{lemma}\label{lemma=multipliers_SLr} Let $\varphi\colon \mathrm{SL}(2n-3,\R) \rightarrow \C$ be an $\mathrm{SO}(2n-3)$-bi-invariant multiplier of $S^p(L^2(\mathrm{SL}(2n-3,\R)))$, and let $t<u<v \in \R$ with $t+\frac{u}{n-2}+v = 0$. Then for $0<\delta<u-t$, we have
\[|\varphi(D(v,u,t))-\varphi(D(v+\frac{\delta}{n-2},u-\delta,t)) | \leq 2(n-2) C_p e^{-\alpha_p(u-t-\delta)} \|\varphi\|_{M(S^p)}.\]
\end{lemma}
\begin{proof}
By the triangle inequality, we can write
\[|\varphi(D(v,u,t))-\varphi(D(v+\frac{\delta}{n-2},u-\delta,t)) | \leq \sum_{k=1}^{n-2} |\varphi(D_{k-1})  - \varphi(D_k)|,\]
where $D_k$ is the diagonal matrix with $(n-2-k)$ eigenvalues equal to $e^v$, $k$ eigenvalues equal to $e^{v+\frac{\delta}{n-2}}$, one eigenvalue equal to $e^{u-\frac{k}{n-2}\delta}$, and $n-2$ eigenvalues equal to $e^t$. By Lemma \ref{lemma=multipliers_SLn}, for each $k$ the term $|\varphi(D_{k-1})  - \varphi(D_k)|$ is less than $2C_p e^{-\alpha_p(u-t-\frac{k}{n-2}\delta)} \|\varphi\|_{M(S^p)}$, which is less than $2C_p e^{-\alpha_p(u-t-\delta)} \|\varphi\|_{M(S^p)}$ and proves the lemma.
\end{proof}
By conjugating by the Cartan automorphism $g \mapsto (g^t)^{-1}$, we get the following.
\begin{lemma}\label{lemma=multipliers_SLr_bis} Let $\varphi\colon \mathrm{SL}(2n-3,\R) \rightarrow \C$ be an $\mathrm{SO}(2n-3)$-bi-invariant multiplier of $S^p(L^2(\mathrm{SL}(2n-3,\R)))$, and let $t<u<v \in \R$ with $t+\frac{u}{n-2}+v = 0$. Then for $0<\delta<v-u$, we have
\[|\varphi(D(v,u,t))-\varphi(D(v,u+\delta,t-\frac{\delta}{n-2})) | \leq 2(n-2) C_p e^{-\alpha_p(v-u-\delta)} \|\varphi\|_{M(S^p)}.\]
\end{lemma}
\begin{proof}[Proof of Proposition \ref{prop=invariant_multiplier_estimate}]
By combining the above two lemmas, we get that for $0<\delta<v$,
\[ |\varphi(D(v,0,-v)) - \varphi(D(v+\frac{\delta}{n-2},0,-v-\frac{\delta}{n-2}))| \leq 4(n-2) C_p e^{-\alpha_p (v-\delta)} \|\varphi\|_{M(S^p)}.\]
This implies that $\varphi(D(t,0,-t))$ satisfies the Cauchy criterion, and hence has a limit. Indeed, for $t \leq s$, define $\delta = \frac{t}{2}$ and the sequence $(v_k)$ by $v_0=t$ and $v_{k+1} = \min(s,v_k+\frac{\delta}{n-2})$ (so that $v_k = \min\{s,(1+\frac{k}{2n-4})t\}$). If $N$ is the first index such that $v_N=s$, then
\begin{multline*}
|\varphi(D(t,0,-t)) - \varphi(D(s,0,-s))|\leq \sum_{k=0}^{N-1} |\varphi(D(v_k,0,-v_k)) - \varphi(D(v_{k+1},0,-v_{k+1}))| \\ \leq \sum_{k=0}^{N-1} 4(n-2) C_p e^{-\alpha_p (\frac 1 2 + \frac{k}{2n-4})t} \leq 4(n-2)C_p \frac{e^{-\alpha_p \frac{t}{2}}}{1-e^{-\alpha_p \frac{t}{2n-4}}}.
\end{multline*}
This proves Proposition \ref{prop=invariant_multiplier_estimate}.
\end{proof}
We can now generalize our result to higher rank simple Lie groups. We will assume that the real rank is at least $9$, so that we only need to consider the classical Lie groups, since all exceptional Lie groups have real rank $8$ or less.
\begin{lemma}\label{lemma=dynkin}
Let $G$ be a connected simple real Lie group with real rank $N \geq 9$. Then $G$ contains a connected closed subgroup $H$ locally isomorphic to $\mathrm{SL}(N,\mathbb{R})$.
\end{lemma}
\begin{proof}
  Since we assume the real rank of $G$ to be at least $9$, the group $G$ is a classical Lie group. By Dynkin's classification of regular semisimple Lie subalgebras of semisimple Lie algebras, it is known that every simple Lie algebra of rank $N \geq 9$ contains $\mathfrak{sl}_N$ as a Lie subalgebra \cite{dynkin} (see \cite{dynkinselectedworks} for a translation). This Lie subalgebra gives rise to a connected Lie subgroup $H$ in $G$ that is locally isomorphic to $\mathrm{SL}(N,\mathbb{R})$, and by a result of Mostow, it is closed, since $G$ has discrete center \cite[last theorem on p.~614]{mostow} (see also \cite[Corollary 1]{dorofaeff}).
\end{proof}
\begin{thm} \label{thm=mainthmlp_generalgroups}
  Let $n \geq 6$, let $N \geq 2n-3$, and let $\Gamma$ be a lattice in a connected simple real Lie group with real rank at least $N$. Then the noncommutative $L^p$-space $L^p(L(\Gamma))$ does not have the CBAP for $p \in \left[1,2-\frac 2 n\right) \cup \left(2+\frac 2 {n-2},\infty\right]$.
\end{thm}
\begin{proof} 
Let $H \subset G$ be a closed subgroup locally isomorphic to $\mathrm{SL}(N,\R)$ given by Lemma \ref{lemma=dynkin}. Then $H$ has finite center, because the fundamental group of $\mathrm{SL}(N,\R)$ is finite. Hence, $H$ does not have the $\mathrm{AP}_{p,\mathrm{cb}}^{\textrm{Schur}}$, since $\mathrm{SL}(N,\R)$ does not have it (see Section \ref{subsec=apschur}). Since the $\mathrm{AP}_{p,\mathrm{cb}}^{\textrm{Schur}}$ passes to closed subgroups (see Section \ref{subsec=apschur}), the group $G$ does not have the $\mathrm{AP}_{p,\mathrm{cb}}^{\textrm{Schur}}$ either. The result now follows by applying Lemma \ref{lem=cbapoapapschur}.
\end{proof}
\begin{rem}
Actually, a stronger statement is true, namely, for $\Gamma$ and $p$ as above, the space $L^p(L(\Gamma))$ does not have the operator space approximation property (OAP) (see section \ref{sec=preliminaries} for the definition). By Lemma \ref{lem=cbapoapapschur}, this follows in the same way from Theorem \ref{thm=nonAP_high_rank} as Theorems \ref{thm=mainthmlp} and \ref{thm=mainthmlp_generalgroups}.
\end{rem}

\section{Non-coarse-embeddability of expander families} \label{sec=expanders}
\subsection{Conventions}
In what follows, we only consider real Banach spaces.

Let $G$ be a locally compact group, let $X$ be a Banach space, and let $O(X)$ denote the group of all invertible linear isometries from $X$ to $X$. In what follows, by an isometric representation of $G$ on $X$, we mean a homomorphism $\pi \colon G \to O(X)$ that is strongly continuous, i.e., the map $G \to X$ given by $g \mapsto \pi(g)x$ is continuous for every $x \in X$. It is known that strong continuity for such representations is equivalent to other forms of continuity \cite[Section 3.3]{monod} (see also \cite[Lemma 2.4]{BFGM}).

If $\pi \colon G \to B(X)$ is a strongly continuous representation of a locally compact group $G$ on a Banach space $X$ and $m$ is a compactly supported signed Borel measure on $G$, then by $\pi(m)$ we denote the operator in $B(X)$ given by
\[
  \pi(m)\xi = \int_G \pi(g)\xi dm(g) \; \textrm{(Bochner integral)} \; \textrm{for} \, \xi \in X.
\]

\subsection{Versions of property (T) relative to Banach spaces}\label{subsec=BanachT}
Let $X$ be a Banach space. As defined in \cite{BFGM}, a locally compact group $G$ has property (T$_X$) if for every isometric representation $\pi \colon G \to O(X)$, the quotient representation $\pi^{\prime} \colon G \to O(X/X^{\pi(G)})$ does not have almost invariant vectors. Here, $X^{\pi(G)}$ denotes the closed subspace of $X$ consisting of vectors that are fixed by $\pi$.

Let $\mathcal{E}$ be a class of Banach spaces, and let $G$ be a locally compact group. Consider the completion $C_{\mathcal E}(G)$ of $C_c(G)$ with respect to the norm $\sup \|\pi(f)\|_{B(X)}$, where the supremum is taken over all isometric representations of $G$ on a space $X$ in $\mathcal E$. Following \cite[Section 4]{lafforguestrengthenedpropertyt} and \cite[D\'efinition 0.4]{lafforguefastfourier}, we say that a group $G$ has property (T$_{\mathcal{E}}^{\textrm{proj}}$) if $C_{\mathcal E}(G)$ contains an idempotent $p$ such that for all isometric representations $\pi$ of $G$ on a space $X$ in $E$, the operator $\pi(p)$ is a projection on $X^{\pi(G)}$. Concretely, this is equivalent to the existence of a sequence $(m_n)_n$ of compactly supported signed measures on $G$ such that $m_n(G)=1$ for all $n \in \mathbb{N}$ and the sequence $(\pi(m_n))_n$ converges in the norm topology of $B(X)$ and uniformly in $(\pi,X)$ to a projection on $X^{\pi(G)}$. We refer to \cite[Section 2.9]{delasalle1} for the equivalence. Property (T$_{\mathcal{E}}^{\textrm{proj}}$) is what we use in order to prove nonembeddability results for expander families.

It is clear that property (T$_{\mathcal{E}}^{\textrm{proj}}$) implies property (T$_X$) for all $X \in \mathcal{E}$. The converse does not hold in general. Indeed, every group $G$ with Kazhdan's property (T) has property (T$_{L^1(G)}$) \cite{BGM}, but no infinite group has property (T$_{\mathcal{E}}^{\textrm{proj}}$) if $L^1(G) \in \mathcal{E}$. However, if $X$ is superreflexive, property (T$_X$) is equivalent with a ``non-uniform'' version of property (T$_{\mathcal{E}}^{\textrm{proj}}$), as is shown by the following result.
\begin{prop} \label{prop=comparison_T} Let $X$ be a superreflexive Banach space. If a discrete group $G$ has property (T$_X$), then for every isometric representation $\pi \colon G \to O(X)$, there is a projection onto the space $X^{\pi(G)}$ in the norm closure of
\[
  \{\pi(m) \mid m \textrm{ is a compactly supported probability measure on }G\}.
\]
If a locally compact group $G$ has property (T$_{\ell^2(X)}$), then for every isometric representation of $G$, such a projection also exists .
\end{prop}
\begin{proof}
Let $\pi \colon G \to O(X)$ be an isometric representation. By \cite[Proposition 2.3]{BFGM}, we can assume that the norm on $X$ is uniformly convex and uniformly smooth. By \cite[Proposition 2.6]{BFGM}, it follows that $X^{\pi(G)}$ has a $G$-invariant closed complemented subspace $Y$. We will construct a compactly supported probability measure $m$ on $G$ such that $\|\mathrm{Id}+\pi(m)\|_{B(Y)}<2$, where $\mathrm{Id}$ denotes the identity operator on $Y$. Replacing $m$ by the measure $A \mapsto \frac{1}{2}(m(A) + m(A^{-1}))$, we can assume that $m$ is symmetric. For $n \geq 1$, define $m_n=(\frac 1 2 \delta_1 + \frac 1 2 m)^{\ast n}$, which is a compactly supported symmetric probability measure, and the sequence $\pi(m_n)$ converges to a projection onto $X^{\pi(G)}$ as $n \to \infty$.

If $G$ is discrete, by property (T$_X$) we know that $Y$ does not almost have invariant vectors, i.e., there is a finite subset $S \subset G\setminus\{1\}$ and $\varepsilon>0$ such that $\sup_{s \in S} \|s \cdot \xi - \xi\|\geq\varepsilon$ for all unit vectors $\xi \in Y$. By the uniform convexity of $X$, this implies that there is a $\delta>0$ such that $\inf_{s \in S} \| s \cdot \xi + \xi\| <2-\delta$. This implies that $\sum_{s \in S} (\mathrm{Id}+\pi(s))$ has norm less than $2|S|-\delta$ on $Y$. In other words, if $m$ is the uniform probability measure on $S$, then $\|\mathrm{Id}+\pi(m)\|_{B(Y)}<2$. 

If $G$ is locally compact, similarly by property (T$_{\ell^2(X)}$) there is $\delta>0$ and a compact subset $Q \subset G$ such that $\inf_{s \in Q} \|s \cdot\xi + \xi\|< 2-\delta$ for all unit vector $\xi \in \ell^2(Y)$. Consider the convex hull $C \subset C(Q)$ of the functions of the form $s \mapsto \langle s\cdot \xi + \xi,\eta\rangle$ for $\xi$ and $\eta$ in the unit balls of $Y$ and $Y^*$, respectively. If $g = \sum \lambda_i \langle s\cdot \xi_i + \xi_i,\eta_i\rangle \in C$, we can write $g = \langle s \cdot \underline \xi+ \underline \xi,\underline \eta\rangle$ for the unit vectors $\underline \xi=(\sqrt{\lambda_i} \xi_i)_i \in \ell^2(Y)$ and $\underline \eta=(\sqrt{\lambda_i} \eta_i)_i \in \ell^2(Y^*)$ and deduce that $\inf_Q g < 2-\delta$. Hence, by the Hahn-Banach Theorem, there is a probability measure $m$ on $Q$ such that $\int g dm \leq 2-\delta$ for all $g \in C$. This implies that $\|\mathrm{Id}+\pi(m)\|_{B(Y)}\leq 2-\delta$.
\end{proof}

\subsection{On the geometry of Banach spaces} \label{subsec=geometrybanachspaces} We now give some background on condition \eqref{eq=d_kX_grows_slowly}. We refer to \cite{tomczakBook} for more information. Then we derive some consequences of \eqref{eq=d_kX_grows_slowly} and Proposition \ref{prop=Schatten_class_estimates}.

Two Banach spaces $X,Y$ are said to be $C$-isomorphic if there is an isomorphism $u\colon X \rightarrow Y$ such that $\|u\| \|u^{-1}\|\leq C$. The infimum of such $C$ is called the Banach-Mazur distance between $X$ and $Y$, also called the isomorphism constant from $X$ to $Y$, and is denoted by $d(X,Y)$. It is known that if $X$ has dimension $k$, then $d(X,\ell^2_k) \leq k^{\frac{1}{2}}$ (indeed, this follows by considering \emph{John's ellipsoid}). We have equality for $X=\ell^1_k$, i.e., $d(\ell^1_k,\ell^2_k) = k^{\frac{1}{2}}$ for all $k \geq 1$.

For a Banach space $X$, put $d_k(X) = \sup \{ d(E,\ell^2_{\dim E}) \mid E\subset X,\,\dim E\leq k\}$. From the reminders above, the inequality $d_k(X)\leq k^{\frac{1}{2}}$ holds for every Banach space $X$ and every $k$, and if $\ell^1$ is finitely representable in $X$, i.e., $X$ contains subspaces $(1+\varepsilon)$-isomorphic to $\ell^1_k$ for every $\varepsilon>0$ and every $k$, then $d_k(X) = k^{\frac{1}{2}}$. Milman and Wolfson \cite{milmanwolfson} proved the converse: if $\limsup_k d_k(X)k^{-\frac{1}{2}} >0$, then $\ell^1$ is finitely representable in $X$. A consequence is that if $d_k(X) < k^{\frac{1}{2}}$ for some $k$, then $\lim_k d_k(X)k^{-\frac{1}{2}}=0$.

The Banach spaces in which $\ell^1$ is not finitely representable were historically called $B$-convex spaces and have been extensively studied. A superreflexive space, i.e., a Banach space $X$ such that every Banach space finitely representable in $X$ is reflexive, is clearly $B$-convex. In fact, a Banach space $X$ is $B$-convex if and only if it has type $>1$  if and only if it is $K$-convex if and only if $\lim_k d_k(X)k^{-\frac{1}{2}}=0$. The rate of convergence to $0$ of this quantity is not yet completely understood. It is known that it converges to $0$ at least as fast as a power of $\log k$, and it is a well-known open problem whether it converges as a power of $k$, as in \eqref{eq=d_kX_grows_slowly} (see \cite[Problem 27.6]{tomczakBook}). Also, it is known that \eqref{eq=d_kX_grows_slowly} holds if $X$ has type $p$ and cotype $q$ satisfying $\frac 1 p - \frac 1 q<\frac 1 2$, which is particularly the case if $X$ has type $2$.

Following \cite{pisierSAF14}, for a Banach space $X$ and a $k \in \mathbb{N}$, one sets
\[ e_k(X) = \sup  \|u \otimes \mathrm{Id}_X\|_{B(\ell^2(X))},\]
where the supremum is taken over all linear maps $u\colon \ell^2 \rightarrow \ell^2$ of norm $1$ and rank $k$. We will need a result by Pietsch, asserting that for a Banach space $X$ and $\beta \leq \frac{1}{2}$, we have $\sup_k d_k(X)k^{-\beta}<\infty$ if and only if $\sup_k e_k(X)k^{-\beta} < \infty$ \cite{pietsch}. In fact, a closer relationship exists between the numbers $e_k(X)$ and $d_k(X)$:
\begin{equation}\label{eq=en_and_BMdistance}
d_k(X) \leq e_k(X) \leq 2 d_k(X).
\end{equation}
The first inequality is classical (see for example \cite[\S 27]{tomczakBook}), whereas the second inequality has been long known to Pisier as a consequence of \cite{tomczak}. With his kind permission, we include a proof here.
\begin{thm}[Pisier] \label{thm=pisier}
Let $X$ be a real Banach space. For every $k \in \mathbb{N}$, 
\[ e_k(X) \leq 2 \sup_{a\colon \ell^2_k \rightarrow \ell^2_k, \|a\| \leq 1} \| a \otimes Id_X\|_{\ell^2_k(X) \rightarrow \ell^2_k(X)}.\]
In particular, $e_k(X) \leq 2 d_k(X)$ for all $k \in \mathbb{N}$.
\end{thm}
\begin{proof} We can assume that $X$ is finite-dimensional. Let $C_k(X)$ denote the set of operators $u\colon X \rightarrow \ell_k^2$ of the form
\[ u(x) = \sum_{i=1}^\infty \xi_i(x) b e_i,\]
where the $\xi_i \in X^*$ satisfy $\sum_{i=1}^{\infty} \|\xi_i\|^2 < 1$, the vectors $e_i$ denote the standard orthonormal basis vectors of $\ell^2$, and $\|b\colon \ell^2 \rightarrow \ell^2_k\|< 1$ (actually it follows from the self-duality of the $2$-summing norm \cite[Proposition 9.10]{tomczakBook} that $C_k(X)$ is the set of operators $u \colon X \rightarrow \ell_k^2$ of $2$-summing norm less than $1$). Since every operator $a\colon \ell^2 \rightarrow \ell^2$ of norm $1$ and rank $k$ can be written as $a= b^* c$ for operators $b,c\colon \ell^2 \rightarrow \ell^2_k$ of norm $1$, we have
\begin{eqnarray*} e_k(X)&=& \sup_{\|\xi\|_{\ell^2(X)} < 1,\|\eta\|_{\ell^2(X^*)}< 1, \|b,c\colon \ell^2 \rightarrow \ell^2_k\| < 1} \langle (c \otimes Id_X)(\eta),(b \otimes Id_{X^*})(\xi)\rangle\\ &=& \sup_{u\in C_k(X), v \in C_k(X^*)} tr( u v^*).\end{eqnarray*}
Similarly, let $D_k(X)$ denote the set of operators $u\colon X \rightarrow \ell_k^2$ of the form
\[ u(x) = \sum_{i=1}^k \xi_i(x) b e_i,\]
where the $\xi_i \in X^*$ satisfy $\sum_{i=1}^k \|\xi_i\|^2 < 1$, the vectors $e_i$ denote the standard orthonormal basis vectors of $\ell_k^2$, and $\|b\colon \ell^2_k \rightarrow \ell^2_k\|< 1$. It follows that
\[ \sup_{a\colon \ell^2_k \rightarrow \ell^2_k, \|a\| \leq 1} \| a \otimes Id_X\|_{\ell^2_k(X) \rightarrow \ell^2_k(X)} = \sup_{u\in D_k(X), v \in D_k(X^*)} tr( u v^*).\]

We claim that $C_k(X) \subset \sqrt 2 \textrm{conv}(D_k(X))$. The claim clearly implies the theorem, and it is proved by duality. The polar of $C_k(X)$ coincides, with respect to the duality $\langle u,v\rangle =tr(uv)$, with the operators $v\colon \ell^2_k \rightarrow X$ such that 
\[\pi_2(v) = \sup_{\|b\colon \ell^2 \rightarrow \ell^2_k\|\leq 1}\left(\sum_i \|u b e_i\|^2\right)^{\frac{1}{2}}\leq 1.\]
Similarly, the polar of $D_k(X)$ is the set of operators $v\colon \ell^2_k \rightarrow X$ such that 
\[\pi_2^{(k)}(v) = \sup_{\|b\colon \ell^2_k \rightarrow \ell^2_k\|\leq 1}\left(\sum_{i=1}^k \|u b e_i\|^2\right)^{\frac{1}{2}}\leq 1.\]
The claim therefore follows from \cite[Theorem 18.4]{tomczakBook}, in which the inequality $\pi_2(v) \leq \sqrt 2\pi_2^{(k)}(v)$ is proved for every rank $k$ linear map.
\end{proof}
The following result is essentially \cite[Proposition 3.3]{delasalle1}. In fact, it is what its proof actually shows.
\begin{prop}\label{prop=link_TX_Schatten_class} Let $X$ be a Banach space such that $\sup_k e_k(X) k^{-\beta}\leq C' < \infty$. Then for every $p< \beta^{-1}$, there is a constant $C_p(X)$ (depending on $C'$, $p$ and $\beta$) such that 
\[\| T \otimes \mathrm{Id}_X\|_{B(L^2(\Omega;X))} \leq C_p(X) \| T\|_{S^p(L^2(\Omega))}\]
for every measure space $(\Omega,\mu)$ and every operator $T\colon L^2(\Omega) \rightarrow L^2(\Omega)$ belonging to the Schatten class $S^p$.
\end{prop}
As a consequence, we obtain the following result.
\begin{lemma}\label{lemma=estimates_Tdelta_X}
Let $n \geq 3$, and let $X$ be a Banach space for which there exist $C'>0$ and $\beta < \frac 1 2(1- \frac{1}{n-1})$ such that $d_k(X) \leq C' k^\beta$ for all $k$. Then there exist $C_X \in \R$ and $\alpha_X>0$ such that for all $\delta,\delta' \in [-1,1]$,
\begin{equation}\label{eq=ine_Tdelta_X} \| (T_{\delta} - T_{\delta'})\otimes \mathrm{Id}_X\|_{B(L^2(\mathrm{SO}(n);X))} \leq C_X \max(|\delta|^{\alpha_X}, |\delta'|^{\alpha_X}).\end{equation}
Moreover, $C_X$ and $\alpha_X$ depend on $C'$ and $\beta$ only.
\end{lemma}
\begin{proof}
By the triangle inequality, we can assume that $\delta'=0$, and we can assume that $\delta \in [-\frac{1}{2},\frac{1}{2}]$ (for $|\delta|\geq \frac{1}{2}$, use that $\|T_\delta\otimes \mathrm{Id}_X\|_{B(L^2(\mathrm{SO}(n);X))} \leq 1$). 

Our assumption on $X$ means (recall \eqref{eq=en_and_BMdistance}) that there is some $\varepsilon>0$ such that $\sup_k e_k(X)  k^{\varepsilon - \frac{1}{q}}<\infty$, where $q=2+\frac{2}{n-2}$. Pick $p>q$ such that $\varepsilon - \frac{1}{q}>-\frac{1}{p}$. Then \eqref{eq=ine_Tdelta_X} follows from Proposition \ref{prop=Schatten_class_estimates} and Proposition \ref{prop=link_TX_Schatten_class}.
\end{proof}
\begin{rem} Since \eqref{eq=ine_Tdelta_X} behaves well with respect to complex interpolation (see for example \cite[Lemma 3.1]{delasalle1}), Lemma \ref{lemma=estimates_Tdelta_X} also holds if $X$ is isomorphic to a complex-interpolation space $[X_0,X_1]_\theta$ for some $0<\theta<1$, where $X_0$ is a space as in the lemma and $X_1$ is an arbitrary Banach space. It might be that the spaces obtained in this way for a fixed $n$ are all the spaces satisfying \eqref{eq=d_kX_grows_slowly}.
\end{rem}
Now we combine Lemma \ref{lemma=estimates_Tdelta_X} with Fell's absorption principle for isometric representations of compact groups on Banach spaces (see \cite[Proposition 2.7]{delasalle1}), asserting that if $\pi \colon K \to O(X)$ is an isometric representation of a compact group $K$ on a Banach space $X$, then for every signed Borel measure $m$ on $K$, we have $\|\pi(m)\|_{B(X)} \leq \|\lambda(m)\|_{B(L^2(K;X))}$. Using the explicit expression for $T_{\delta}$, this implies the following result.
\begin{lemma}\label{lemma=coefficients_Banach_rep_SOn} Let $n \geq 3$, and let $X$ be a Banach space satisfying \eqref{eq=ine_Tdelta_X} for some $C_X$, $\alpha_X>0$ and all $\delta,\delta' \in [-1,1]$. Then for every linear isometric representation $\pi$ of $\mathrm{SO}(n)$ on $X$ and all $\mathrm{SO}(n-1)$-invariant unit vectors $\xi\in X$ and $\eta \in X^*$, the function $\varphi(g) = \langle \pi(g) \xi,\eta\rangle$ satisfies
\[ |\varphi(g) - \varphi(g')| \leq C_X \max(|g_{1,1}|^{\alpha_X}, |g'_{1,1}|^{\alpha_X}) \|\xi\|_X \|\eta\|_{X^*}\]
for all $g,g' \in \mathrm{SO}(n)$.
\end{lemma}

\subsection{Explicit behaviour of coefficients}
Proposition \ref{prop=invariant_multiplier_estimate} followed from Lemma \ref{lemma=mutlipliers_SOn} with a proof of combinatorial nature. The only property of $S^p$-multipliers that was used is the following: 
if $\varphi\colon G \rightarrow \C$ is an $S^p$-multiplier on a locally compact group $G$ and $H$ is a closed subgroup of $G$ and $g,g' \in G$, then the function on $H$ given by $h \mapsto \varphi(ghg')$ is an $S^p$-multiplier of the group $H$ with norm not greater than $\|\varphi\|_{M(S^p)}$. The same property holds if, for a given Banach space $X$, the word \emph{$S^p$-multiplier} is replaced by \emph{coefficient of a continuous isometric representation on $X$} and $\|\varphi\|_{M(S^p)}$ is replaced by the norm equal to the infimum of $\|\xi\|_X \|\eta\|_{X^*}$ over all strongly continuous isometric representations $\pi$ of $G$ on $X$ and all vectors $\xi$ and $\eta$ such that $\varphi(g)=\langle \pi(g) \xi,\eta\rangle$. Therefore, we can deduce the following proposition from Lemma \ref{lemma=coefficients_Banach_rep_SOn} with exactly the same proof as for Proposition \ref{prop=invariant_multiplier_estimate}.
\begin{prop} \label{prop=invariant_coefficient_estimate} Let $n \geq 3$, and let $X$ be a Banach space for which there exist $C_X\in \R$ and $\alpha_X>0$ such that \eqref{eq=ine_Tdelta_X} holds for all $\delta,\delta' \in [-1,1]$. Then there is a function $\varepsilon_X \in C_0(\R_+)$ (depending on $C_X$ and $\alpha_X$ only) such that the following holds: for every linear isometric strongly continuous representation $\pi$ of ${\mathrm{SL}(2n-3,\R)}$ on $X$ and all $\mathrm{SO}(2n-3)$-invariant vectors $\xi\in X$ and $\eta \in X^*$, the function $\varphi(g) = \langle \pi(g) \xi,\eta\rangle$ satisfies the fact that $\varphi(D(t,0,-t))$ has a limit $c$ as $t \to \infty$, and 
\[ |\varphi(D(t,0,-t)) - c | \leq \varepsilon_X(t)  \|\xi\|_X \|\eta\|_{X^*}.\]
\end{prop}

\subsection{Non-coarse-embeddability}
Proposition \ref{prop=invariant_coefficient_estimate} leads to the following theorem.
\begin{thm}\label{thm=TBanachiqueG} Let $X$ be a superreflexive Banach space satisfying \eqref{eq=d_kX_grows_slowly}. Then there exists an integer $N \geq 9$ such that for every connected simple Lie group $G$ of real rank at least $N$ the following holds: there is a sequence of compactly supported symmetric probability measures $m_l$ on $G$ such that for every isometric representation $\pi \colon G \to O(X)$, the sequence $(\pi(m_l))$ converges to a projection on $X^{\pi(G)}$. Moreover, the integer $N$ and the measures $m_l$ depend only on the value of $\beta$ in \eqref{eq=d_kX_grows_slowly}.
\end{thm}
\begin{rem}
The above theorem directly implies that for every $\beta < \frac{1}{2}$, there exists an $N \geq 9$ such that every connected simple Lie group with real rank at least $N$ has property (T$_{\mathcal{E}}^{\mathrm{proj}}$) with respect to the class of superreflexive Banach spaces satisfying \eqref{eq=d_kX_grows_slowly} for the given value of $\beta$ and a fixed $C>0$.
\end{rem}
\begin{proof}
By our assumption, there exist $C>0$ and $\beta < \frac 1 2$ such that $d_k(X) \leq C k^\beta$ for all $k$. Take $n$ so that $\beta < \frac 1 2(1- \frac{1}{n-1})$. Our main task is to prove the theorem for $\mathrm{SL}(2n-3,\R)$. Let now $G=\mathrm{SL}(2n-3,\mathbb{R})$ and define a sequence of symmetric probability measures $m_l$ by
\[ m_l(f) = \iint_{\mathrm{SO}(2n-3) \times \mathrm{SO}(2n-3)} f(k D(l,0,-l) k') dk dk',\]
where the notation $D(l,0,-l)$ was introduced in \eqref{eq=def_Duvt}.
Let $\pi$ be an isometric strongly continuous representation of $G$ on $X$. By Proposition \ref{prop=invariant_coefficient_estimate}, $\pi(m_l)$ has a limit $P$ in the norm topology of $B(X)$. We claim that $P$ is a projection on $X^{\pi(G)}$. This is where we use the assumption that $X$ is superreflexive. By \cite[Proposition 2.3]{BFGM}, we can assume that the norm on $X$ is uniformly convex and uniformly smooth. It is clear that $P$ acts as the identity on $X^{\pi(G)}$. By \cite[Proposition 2.6]{BFGM}, $X^{\pi(G)}$ has a $G$-invariant complement closed subspace. By replacing $X$ by this complement subspace, we can assume that $X^{\pi(G)}=0$, and we have to prove that $P=0$. This follows from the version of the Howe-Moore property proved by Shalom (see \cite[Theorem 9.1]{BFGM}).

Now standard arguments (see, e.g., \cite[Section 1.6]{bekkadelaharpevalette} or \cite[Section 4]{lafforguestrengthenedpropertyt}) imply that the conclusion of the theorem holds for every connected simple real Lie group containing a closed subgroup locally isomorphic to $\mathrm{SL}(2n-3,\R)$, and hence, by Lemma \ref{lemma=dynkin}, for every simple real Lie group of real rank $\geq \max\{9,2n-3\}$.
\end{proof}
\begin{rem}\label{rem=TBanachiquebis} If $X$ satisfies \eqref{eq=d_kX_grows_slowly} but is not superreflexive, the conclusion of the above theorem still holds for the representations of the form $\pi \otimes 1_X$ on $L^2(\Omega,\mu;X)$, where $\pi$ is a unitary representation on $L^2(\Omega,\mu)$ that comes from a measure-preserving action of $G$ on a $\sigma$-finite measure space $(\Omega,\mu)$.

Indeed, as in the proof above, it is sufficient to show that if $n$ is such that $\beta < \frac{1}{2}(1-\frac{1}{n-1})$, any representation of $\mathrm{SL}(2n-3,\mathbb{R})$ of the form $\pi \otimes 1_X$ as above satisfies that $(\pi \otimes 1_X)(m_l)=\pi(m_l) \otimes 1_X$ converges in the norm topology to a projection onto the $\mathrm{SL}(2n-3,\mathbb{R})$-invariant vectors. On the one hand, by Fubini's theorem, $e_k(L^2(\Omega;X))=e_k(X)$ for all $k$, and hence, by \eqref{eq=en_and_BMdistance}, $L^2(\Omega,\mu;X)$ satisfies \eqref{eq=d_kX_grows_slowly} for the same $\beta$ as $X$. Proposition \ref{prop=invariant_coefficient_estimate} then implies that $(\pi \otimes 1_X)(m_l)$ has a limit in the norm topology. On the other hand, by the Howe-Moore property for (unitary representations of) $\mathrm{SL}(2n-3,\mathbb{R})$, the sequence $\pi(m_l)$ converges in the weak operator topology to the orthogonal projection $P$ on the $G$-invariant vectors on $L^2(\Omega,\mu)$, which shows that the limit of $\pi(m_l) \otimes 1_X$ is $P \otimes 1$, a projection onto the $\mathrm{SL}(2n-3,\mathbb{R})$-invariant vectors.
\end{rem}

By the argument recalled in Section \ref{subsec=coarse_embedd}, Theorem \ref{thm=TBanachiqueG} (if $X$ is superreflexive) or Remark \ref{rem=TBanachiquebis} (if $X$ not superreflexive) imply the following result.
\begin{thm}\label{thm=expanders_general} Let $X$ be a Banach space satisfying \eqref{eq=d_kX_grows_slowly}. Then there exists a natural number $N \geq 9$ such that if $\Gamma$ is a lattice in a connected simple real Lie group of real rank at least $N$, if $(\Gamma_i)_{i \in \mathbb{N}}$ is a sequence of finite-index subgroups of $\Gamma$ such that $|\Gamma / \Gamma_i| \to \infty$ for $i \to \infty$ and if $S$ is a symmetric finite generating set of $\Gamma$, then the sequence of Schreier graphs $(\Gamma/\Gamma_i,S)$ does not coarsely embed into $X$.
\end{thm}
Theorem \ref{thm=expanders} follows as a particular case of Theorem \ref{thm=expanders_general}.

\begin{rem}
(Remark added after publication.) By using Veech's version of the Howe-Moore property \cite{veech}, which holds in the more general setting of reflexive Banach spaces rather than superreflexive ones, the superreflexivity assumption in Theorems 1.4 and 5.8 can be replaced by the assumption that $X$ is reflexive.
\end{rem}

\end{document}